\documentclass[12pt]{amsproc}
\usepackage{ytableau,hyperref,amsaddr,amssymb}
\title{Tableau Correspondences and Representation Theory}
\author{Digjoy Paul and Amritanshu Prasad}\address{\href{http://www.imsc.res.in}{The Institute of Mathematical Sciences (HBNI) Chennai}}\email{\href{mailto:digjoypaul@gmail.com}{digjoypaul@gmail.com},\href{mailto:amri@imsc.res.in}{amri@imsc.res.in}}\email{\href{mailto:arghya.sadhukhan0@gmail.com}{arghya.sadhukhan0@gmail.com}}
\author{Arghya Sadhukhan}\address{\href{https://www-math.umd.edu}{University of Maryland}}
\newcommand{\gl}[1]{\mathrm{GL}_{#1}(\CC)}
\newcommand{\glw}[1][W]{\mathrm{GL}(#1)}
\newcommand{\ch}{\mathrm{ch}}
\newcommand{\CC}{\mathbf C}
\newcommand{\ZZ}{\mathbf Z}
\newcommand{\tr}{\mathrm{tr}}
\newcommand{\Tab}{\mathrm{Tab}}
\newcommand{\wt}{\mathrm{wt}}
\newcommand{\M}{\mathbf{M}}
\newcommand{\N}{\mathbf{N}}
\newcommand{\RSK}{\mathrm{RSK}}
\newcommand{\BUR}{\mathrm{BUR}}
\newcommand{\Sym}{\mathrm{Sym}}
\newcommand{\TP}{\mathrm{TP}}
\newcommand{\Rep}{\mathrm{Rep}}
\newcommand{\Ind}{\mathrm{Ind}}
\newtheorem{theorem}{Theorem}
\theoremstyle{definition}
\newtheorem{definition}[theorem]{Definition}
\theoremstyle{remark}
\newtheorem{example}[theorem]{Example}
\newtheorem{remark}[theorem]{Remark}
\begin{document}
\begin{abstract}
  We deduce decompositions of natural representations of general linear groups and symmetric groups from combinatorial bijections involving tableaux.
  These include some of Howe's dualities, Gelfand models, the Schur-Weyl decomposition of tensor space, and multiplicity-free decompositions indexed by threshold partitions.
\end{abstract}
\maketitle
\section{Introduction}
\label{sec:introduction}
Bijections in algebraic combinatorics are often manifestations of results in representation theory.
For example, the Robinson-Schensted correspondence, between permutations and pairs of standard tableaux of the same shape, reflects the fact that the sum of squares of dimensions of irreducible representations of a symmetric group is its order.
Sometimes combinatorial identities can be used to \emph{prove} results in representation theory.
For instance, the classification of the irreducible representations of symmetric groups and Young's rule are deduced from the Robinson-Schensted-Knuth (RSK) correspondence in \cite{rtcv}.
The dual RSK correspondence is used to determine what happens when a representation is twisted by the sign character.
This article collects many more instances of this phenomenon.

Section~\ref{sec:semist-young-tabl} reviews basic facts about semistandard Young tableaux and Schur functions.
Section~\ref{sec:gln-chars} reviews fundamental results in the polynomial representation theory of general linear groups.
The RSK correspondence is a bijection from matrices with non-negative integer entries onto pairs of semistandard tableaux of the same shape.
The dual RSK correspondence is a bijection from matrices with entries $0$ or $1$ onto pairs of semistandard tableaux of mutually conjugate shape.
In Section~\ref{sec:rsk-corr-its}, the RSK correspondence and its dual are shown to imply the $(\mathrm{GL}_m, \mathrm{GL}_n)$-duality and the skew $(\mathrm{GL}_m, \mathrm{GL}_n)$-duality theorems of Howe.
The symmetry property of the RSK correspondence gives a Gelfand model for $\gl n$.
Sch\"utzenberger's lemma on the RSK correspondence is shown to give a refinement of the Gelfand model.

The Burge correspondence is another bijection from matrices with non-negative integer entries onto pairs of semistandard tableaux of the same shape, less well-known than the RSK correspondence.
In Section~\ref{sec:burg-corr}, we show that its symmetry property gives rise to another Gelfand model for $\gl n$, which also has a refinement based on an adaptation of Sch\"utzenberger's lemma to this setting.
Two more correspondences of Burge give rise to multiplicity-free decompositions into representations parameterized by threshold partitions, and conjugate threshold partitions. These are discussed in Section~\ref{sec:two-more}.

Section~\ref{sec:repr-symm-groups} describes a passage from representations of $\gl n$ to representations of $S_n$ via the all-ones weight space.
Applying this to one side of $(\mathrm{GL}_m, \mathrm{GL}_n)$-duality, we recover Schur-Weyl duality.
Applying it to the Gelfand models of $\gl n$, we recover combinatorial Gelfand models for $S_n$ due to Inglis, Richardson, and Saxl~\cite{Inglis1990}.
Applying it to the multiplicity-free representations of $\gl n$ in Section~\ref{sec:two-more}, we recover some other interesting multiplicity-free representations of $S_n$, which were outlined by Bump~\cite{bump} using completely different methods.

Most of the results in this article have appeared in the MSc thesis of Arghya Sadhukhan~\cite{arghya-thesis}.
\section{Semistandard Young Tableaux and Schur Polynomials}
\label{sec:semist-young-tabl}
In this section we recall basic facts about semistandard Young tableaux and Schur polynomials.
For detailed, self-contained expositions, see \cite{rtcv,schur}.
\begin{definition}
  A \emph{semistandard Young tableau} in $n$ letters is a left-justified array of boxes with rows of weakly decreasing length filled with numbers between $1$ and $n$ such that
  \begin{enumerate}
  \item the numbers increase weakly from left to right along rows,
  \item the numbers increase strictly from top to bottom along columns.
  \end{enumerate}
  If a semistandard Young tableau has $l$ rows, and $\lambda_i$ boxes in the $i$th row for $i=1,\dotsc,l$, then the partition $\lambda=(\lambda_1,\dotsc,\lambda_l)$ is called the shape of the tableau.
  The weight of a tableau is the integer vector $(\mu_1,\dotsc,\mu_n)$ where $\mu_i$ the number of times the number $i$ occurs.
  If $t$ is a tableau, we write $\wt(t)$ for its weight.
  Write $\Tab_n(\lambda)$ for the set of all semistandard Young tableaux in $n$ letters having shape $\lambda$.
  The subset of $\Tab_n(\lambda)$ consisting of tableaux with weight $\mu$ is denoted by $\Tab(\lambda,\mu)$.
\end{definition}
\begin{example}
  The tableau in $7$ letters
  \begin{displaymath}
    t=\ytableaushort{22355,4446,57}
  \end{displaymath}
  has shape $(5,4,2)$ and $\wt(t)=(0,2,1,3,3,1,1)$.
\end{example}
Given an integer vector $\mu=(\mu_1,\dotsc,\mu_n)$ with $n$-coordinates, let $x^\mu$ denote the monomial $x_1^{\mu_1}\dotsb x_n^{\mu_n}$ in $n$ variables.
\begin{definition}
  [Schur Polynomial]
  For each integer partition $\lambda$, the Schur polynomial in $n$ variables corresponding to $\lambda$ is defined as:
  \begin{displaymath}
    s_\lambda(x_1,\dotsc,x_n) = \sum_{t\in \Tab_n(\lambda)} x^{\wt(t)}.
  \end{displaymath}
\end{definition}
It turns out that $s_\lambda$ is a symmetric polynomial in $n$ variables.
If the partition $\lambda$ has more than $n$ non-zero parts, there are no tableaux in $n$ letters of shape $\lambda$ (because the first column of the tableau has to be strictly increasing, and has length equal the number of parts of $\lambda$).
On the other hand, if the number of non-zero parts of $\lambda$ is at most $n$, then there is at least one such tableau, namely the one whose $i$th row has all boxes filled with $i$.
Therefore $s_\lambda(x_1,\dotsc,x_n)=0$ if and only if $\lambda$ has more than $n$ non-zero parts.

Let $\Lambda^d_n$ denote the set of all integer partitions  of $d$ with at most $n$ parts.
By padding a partition with $0$'s, write it as an integer vector $\lambda=(\lambda_1,\dotsc,\lambda_n)$ with exactly $n$ parts.
Thus
\begin{displaymath}
  \Lambda^d_n = \{(\lambda_1,\dotsc,\lambda_n)\in \ZZ^n\mid \lambda_1+\dotsb+\lambda_n = d,\; \lambda_1\geq \dotsb \geq \lambda_n\geq 0\}.
\end{displaymath}
\begin{theorem}
  \cite[Theorem~5.4.3]{rtcv}
  The set
  \begin{displaymath}
    \{s_\lambda\mid \lambda\in \Lambda^d_n\}
  \end{displaymath}
  is a basis for the space of all symmetric polynomials of degree $d$ in $n$ variables with coefficients in $\CC$.
\end{theorem}
\section{Polynomial Representations of $\gl n$}
\label{sec:gln-chars}
In this section we recall basic facts about polynomial representations of $\gl n$.
Details can be found in \cite[Chapter~6]{rtcv}.
\begin{definition}
  A \emph{polynomial representation} of $\gl n$ is a finite dimensional vector space $W$ over $\CC$, together with a group homomorphism $\rho:\gl n\to \glw$, such that each entry of the matrix of $\rho(g)$ (with respect to any basis of $W$) is a polynomial in the entries of the matrix $g$.
  The polynomial representation $W$ is said to be \emph{homogeneous of degree $d$} if these polynomials are homogeneous of degree $d$.
  It is said to be \emph{irreducible} if there is no non-trivial proper subspace $W'\subset W$ that is left invariant by $\rho(g)$ for all $g\in \gl n$.
  Representations $(\rho_1,W_1)$ and $(\rho_2,W_2)$ are said to be \emph{isomorphic} if there exists a linear isomorphism $T:W_1\to W_2$ such that $\rho_2(g)\circ T = T\circ \rho_1(g)$ for each $g\in \gl n$.
\end{definition}
\begin{example}
  [The trivial representation]
  The only homogeneous polynomial representations of degree $0$ are where $\rho(g)$ is identically the identity operator of $W$.
\end{example}
\begin{example}
  [Direct sum]
  If $(\rho_1,W_1)$ and $(\rho_2,W_2)$ are polynomial representations of $\gl n$, then $(\rho_1\oplus \rho_2,W_1\oplus W_2)$ is also a polynomial representation of $\gl n$.
  If $\rho_1$ and $\rho_2$ are homogeneous of degree $d$, then so is $\rho_1\oplus\rho_2$.
\end{example}
\begin{example}
  [The defining representation]
  \label{example:defining-rep}
  The identity map $\rho:\gl n\to \gl n$ defines a representation of $\gl n$ on $\CC^n$.
  This representation is homogeneous of degree $1$ and is called the \emph{defining representation}.
\end{example}
\begin{example}
  [Tensor product]
  \label{example:tensor-prod}
  If $(\rho_1,W_1)$ and $(\rho_2,W_2)$ are polynomial representations of $\gl n$, then so is $(\rho_1\otimes \rho_2,W_1\otimes W_2)$.
  If $\rho_1$ has degree $d_1$ and $\rho_2$ has degree $d_2$, then $\rho_1\otimes \rho_2$ has degree $d_1+d_2$.
  As a particular example, $\otimes^d \CC^n$, the $d$-fold tensor power of the defining representation $\CC^n$, is homogeneous of degree $d$ and dimension $n^d$.
\end{example}
\begin{example}
  [Symmetric and alternating tensors]
  The symmetric group $S_d$ acts on $\otimes^d \CC^n$ by permuting the tensor factors.
  The subspace of invariant tensors is denoted by $\Sym^d \CC^n$.
  Since the actions of $\gl n$ and $S_d$ on $\otimes^d \CC^n$ commute, $\Sym^d\CC^n$ is invariant under the action of $\gl n$, forming a homogeneous representation of degree $d$, which is known as the $d$th symmetric tensor.
  Let $L$ denote the subalgebra of $\bigoplus_{d=0}^\infty \otimes^d \CC^n$ generated by $x\otimes x$.
  Let $L_d=L\cap \otimes^d\CC^n$.
  Then $L_d$ is an invariant subspace of $\otimes^d\CC^n$, and the quotient $\wedge^d \CC^n=\otimes^d \CC^n/L_d$ is a representation of $\gl n$, known as the $d$th alternating tensor.

  Let $\lambda=(\lambda_1,\dotsc,\lambda_l)$ be a partition of $d$.
  Then
  \begin{align*}
    \Sym^\lambda \CC^n & = \bigotimes_{i=1}^l \Sym^{\lambda_i} \CC^n\\
    \wedge^\lambda \CC^n & = \bigotimes_{i=1}^l \wedge^{\lambda_i}\CC^n
  \end{align*}
  are homogeneous polynomial representations of $\gl n$ of degree $d$.
\end{example}
Let $T_n\subset \gl n$ denote the subgroup of invertible diagonal matrices.
We use $\Delta(x_1,\dotsc,x_n)$ to denote the diagonal matrix with entries $x_1,\dotsc, x_n$.
\begin{definition}
  [Weight vector]
  \label{definition:weight-vector}
  Let $(\rho,W)$ be a polynomial representation of $\gl n$.
  A vector $v\in W$ is said to be a \emph{weight vector with weight $\mu=(\mu_1,\dotsc,\mu_n)$} if, for all $x_1,\dotsc,x_n\in \CC^*$,
  \begin{displaymath}
    \rho(\Delta(x_1,\dotsc,x_n))v = x_1^{\mu_1}\dotsb x_m^{\mu_n}v.
  \end{displaymath}
  The subspace of all weight vectors of weight $\mu$ is called the $\mu$-\emph{weight space} of $W$ and denoted $W(\mu)$.
\end{definition}
\begin{theorem}
  \label{theorem:weight-basis}
  \cite[Theorem~6.6.5]{rtcv}
  Every polynomial representation $(\rho,W)$ of $\gl n$ admits a basis of weight vectors.
\end{theorem}
\begin{example}
  \label{example:weight-basis-defining}
  In the defining representation of $\gl n$ (Example~\ref{example:defining-rep}), the coordinate vector $e_i\in \CC^n$ is a weight vector with weight $e_i$.
  The set $\{e_i\mid i=1,\dotsc,n\}$ is a basis of weight vectors for $\CC^n$.
\end{example}
\begin{example}
  \label{example:weight-basis-tensor}
  In the tensor space $\otimes^d \CC^n$ (Example~\ref{example:tensor-prod}),
  \begin{displaymath}
    \{e_{i_1}\otimes \dotsb \otimes e_{i_d}\mid (i_1,\dotsc,i_d)\in \{1,\dotsc,n\}^d \}
  \end{displaymath}
  is a basis of $\otimes^d \CC^n$ consisting of weight vectors.
  The weight of $e_{i_1}\otimes \dotsb \otimes e_{i_d}$ is $(\mu_1,\dotsc,\mu_n)$, where $\mu_k$ is the number of $k$ for which $i_k=i$.
\end{example}
\begin{example}
  For non-negative integers $n$ and $d$, define
  \begin{displaymath}
    M(n,d) = \{(i_1,\dotsc,i_d)\mid 1\leq i_1\leq \dotsb \leq i_d\leq n\}.
  \end{displaymath}
  Given $I=(i_1,\dotsc,i_d)\in M(n,d)$ the symmetric tensor $e_{i_1}\dotsb e_{i_d}$ is a weight vector in $\Sym^d\CC^n$ with weight $(\mu_1,\dotsc,\mu_n)$ where $\mu_i$ is the number of $k$ for which $i_k=i$.
  As $I$ runs over $M(n,d)$, these vectors form a basis of $\Sym^d\CC^n$.

  For non-negative integers $n$ and $d$, define
  \begin{displaymath}
    N(n,d) = \{(i_1,\dotsc,i_d)\mid 1\leq i_1< \dotsb < i_d\leq n\}.
  \end{displaymath}
  Given $I=(i_1,\dotsc,i_d)\in N(n,d)$ the alternating tensor $e_{i_1}\wedge\dotsb \wedge e_{i_d}$ is a weight vector in $\wedge^d\CC^n$ with weight $(\mu_1,\dotsc,\mu_n)$ where $\mu_i$ is the number of $k$ for which $i_k=i$.
  As $I$ runs over $N(n,d)$, these vectors form a basis of $\wedge^d\CC^n$.
\end{example}
Most proofs in this article will be based on finding bases of weight vectors of representations as in the examples above.
\begin{definition}
  The \emph{character} of a polynomial representation $W$ of $\gl n$ is defined as
  \begin{displaymath}
    \ch_W(x_1,\dotsc,x_n) = \tr(\rho(\Delta(x_1,\dotsc,x_n)); W).
  \end{displaymath}
  For any permutation $w\in S_n$, $\Delta(x_{w(1)},\dotsc,x_{w(n)})$ is conjugate to\linebreak $\Delta(x_1,\dotsc,x_n)$.
  Therefore $\ch_W$ is a symmetric polynomial.
  If $W$ is homogeneous of degree $d$, then $\ch_W$ is homogeneous of degree $d$.
\end{definition}
\begin{remark}
  Theorem~\ref{theorem:weight-basis} implies that
  \begin{displaymath}
    \ch_W(x_1,\dotsc,x_n) = \sum_{\mu} \dim W(\mu) x^\mu,
  \end{displaymath}
  the sum being over all $n$-tuples $\mu$ of non-negative integers.
\end{remark}
Let $\lambda'$ denote the partition \emph{conjugate} to $\lambda$, namely,
\begin{displaymath}
  \lambda'_i = \#\{j\mid \lambda_j\geq i\} \text{ for } i=1,\dotsc,n.
\end{displaymath}
\begin{theorem}
  \cite[Chapter~6]{rtcv}
  Let $\rho:\gl n\to \glw$ be a polynomial representation of $\gl n$.
  Then we have:
  \begin{enumerate}
  \item For every partition $\lambda\in \Lambda^d_n$, there exists an irreducible homogeneous polynomial representation $W^n_\lambda$ of $\gl n$ of degree $d$ which occurs in both $\Sym^\lambda \CC^n$ and $\wedge^{\lambda'} \CC^n$.
    This representation satisfies:
    \begin{displaymath}
      \ch_{W^n_\lambda}(x_1,\dotsc,x_n) = s_\lambda(x_1,\dotsc,x_n),
    \end{displaymath}
    the Schur polynomial corresponding to $\lambda$, in $n$ variables.
  \item Every polynomial representation $W$ of $\gl n$ has a unique decomposition of the form:
    \begin{displaymath}
      W = \bigoplus_\lambda (W^n_\lambda)^{\oplus m_\lambda}
    \end{displaymath}
    into irreducible polynomial representations (since $W$ is finite dimensional, it should be understood that $m_\lambda$ is positive for only finitely many $\lambda$).
  \item Two polynomial representations of $\gl n$ are isomorphic if and only if their characters are equal.
  \end{enumerate}
\end{theorem}
\section{The RSK Correspondence and its Dual}
\label{sec:rsk-corr-its}
Let $\mu=(\mu_1,\dotsc,\mu_m)$ and $\nu=(\nu_1,\dotsc,\nu_n)$ be non-negative integer vectors with common sum $d$.
Let $\M_{\mu\nu}$ denote the set of all $m\times n$ matrices $A = (a_{ij})$ with non-negative integer entries such that
\begin{displaymath}
  \sum_j a_{ij} = \mu_i \text{ for }i=1,\dotsc,m \text{ and } \sum_i a_{ij} = \nu_j \text{ for } j=1,\dotsc,n.
\end{displaymath}
The RSK correspondence \cite[Section~3]{MR0272654} is an algorithmic bijection:
\begin{displaymath}
  \RSK : \M_{\mu\nu} \tilde\to \coprod_{\lambda\vdash d} \Tab(\lambda,\nu)\times \Tab(\lambda,\mu).
\end{displaymath}
\begin{theorem}
  \label{theorem:rsk}
  Let $W = \Sym^d(\CC^m\otimes \CC^n)$.
  Viewing $\CC^m$ and $\CC^n$ as the defining representations of $\gl m$ and $\gl n$ respectively, the functorial nature of $\Sym^d$ implies that $W$ is a representation of $\gl m\times \gl n$.
  The decomposition of $W$ into irreducible representations of $\gl m\times \gl n$ is given by:
  \begin{displaymath}
    W = \bigoplus_{\lambda\vdash d} W^n_\lambda\otimes W^m_\lambda.
  \end{displaymath}
\end{theorem}
\begin{proof}
  Let $e_1\dotsc, e_n$ and $f_1,\dotsc,f_m$ denote the coordinate vectors in $\CC^n$ and $\CC^m$ respectively.
  Let $\mu$ and $\nu$ be non-negative integer vectors with common sum $d$.
  Given $A\in \M_{\mu\nu}$, define a vector $v_A\in W$ by
  \begin{displaymath}
    v_A = \prod_{i=1}^n\prod_{j=1}^m (e_i\otimes f_j)^{a_{ij}}.
  \end{displaymath}
  Then $v_A$ is an eigenvector for the action of $\Delta(x_1,\dotsc,x_n)\in T_n$, and $\Delta(y_1,\dotsc,y_m)\in T_m$ with eigenvalue $x^\nu y^\mu$.
  Therefore
  \begin{align*}
    \tr(\Delta(x)\Delta(y); W) & = \sum_\nu \sum_\mu \sum_{A\in \M_{\mu\nu}}x^\nu y^\mu\\
                               & = \sum_\nu\sum_\mu\sum_\lambda \sum_{t'\in \Tab(\lambda,\nu)}\sum_{t''\in \Tab(\lambda,\mu)} x^{\wt(t')}y^{\wt(t'')}\\
                               & = \sum_\lambda s_\lambda(x)s_\lambda(y)\\
                               & = \sum_{\lambda} \ch_{W^n_\lambda}(x)\ch_{W^m_\lambda}(y)\\
                               & = \sum_\lambda \ch_{W^n_\lambda\otimes W^m_\lambda}(x, y).
  \end{align*}
\end{proof}
\begin{remark}
  Theorem~\ref{theorem:rsk} is called \emph{$(\mathrm{GL}_n, \mathrm{GL}_m)$-duality} by Howe~\cite[Theorem~2.1.2]{MR1321638}, who gives a different proof. See also Benson and Ratcliff~\cite[Theorem~4.1.1]{doi:10.1142/9789812562500_0006}.
\end{remark}
Let $\mu=(\mu_1,\dotsc,\mu_m)$ and $\nu=(\nu_1,\dotsc,\nu_n)$ be non-negative integer vectors with common sum $d$.
Let $\N_{\mu\nu}$ denote the set of all $m\times n$ matrices $A = (a_{ij})$ with entries in $\{0,1\}$ such that
\begin{displaymath}
  \sum_j a_{ij} = \mu_i \text{ for }i=1,\dotsc,m \text{ and } \sum_i a_{ij} = \nu_j \text{ for } j=1,\dotsc,n.
\end{displaymath}
The dual RSK correspondence \cite[Section~5]{MR0272654} is an algorithmic bijection:
\begin{displaymath}
  \RSK^*: \N_{\mu\nu}\tilde\to \coprod_{\lambda\vdash d} \Tab(\lambda',\nu)\times \Tab(\lambda,\mu).
\end{displaymath}
\begin{theorem}
  \label{theorem:dual-rsk}
  Let $W=\wedge^d(\CC^m\otimes \CC^n)$.
  The decomposition of $W$ into irreducible representations of $\gl m\times \gl n$ is given by:
  \begin{displaymath}
    W = \bigoplus_\lambda W_{\lambda'}^n\otimes W_\lambda^m.
  \end{displaymath}
\end{theorem}
\begin{proof}
  Given $A$ in $\N_{\mu\nu}$ define $v_A\in W$ by:
  \begin{displaymath}
    v_A = \bigwedge_{i=1}^n\bigwedge_{j=1}^m (e_i\otimes f_j)^{a_{ij}}.
  \end{displaymath}
  Then $v_A$ is an eigenvector for the action of $\Delta(x_1,\dotsc,x_n)\in T_n$ and $\Delta(y_1,\dotsc,y_m)\in T_m$ with eigenvalue $x^\nu y^\mu$.
  Now, proceeding as in the proof of Theorem~\ref{theorem:rsk} gives the result.
\end{proof}
\begin{remark}
  Theorem~\ref{theorem:dual-rsk} is called skew \emph{$(\mathrm{GL}_n, \mathrm{GL}_m)$-duality} by Howe \cite[Theorem~4.1.1]{MR1321638}, who gives a different proof.
\end{remark}
Knuth proved a symmetry theorem for the RSK correspondence:
\begin{displaymath}
  \RSK(A)=(P,Q) \text{ if and only if } \RSK(A') = (Q,P).
\end{displaymath}
Here $A'$ denotes the transpose of the matrix $A$.
Therefore, if $A$ is symmetric, then $\RSK(A)$ is of the form $(P,P)$ for some semistandard Young tableau $P$.
Let $\M^{\mathrm{sym}}_{\nu\nu}$ denote the set of symmetric $n\times n$ matrices in $\M_{\nu\nu}$.
The symmetry property of the RSK correspondence implies that it induces a bijection:
\begin{equation}
  \label{eq:rsk-sym}
  \RSK: \M^{\mathrm{sym}}_{\nu\nu} \tilde\to \coprod_{\lambda\vdash d} \Tab(\lambda,\nu).
\end{equation}
\begin{theorem}
  \label{theorem:rsk-symmetry}
  Let $W=\bigoplus_{k+2l=d}\Sym^k(\CC^n)\otimes \Sym^l(\wedge^2 \CC^n)$.
  This representation of $\gl n$ has decomposition into irreducibles given by:
  \begin{displaymath}
    W = \bigoplus_{\lambda\vdash d} W_\lambda^n.
  \end{displaymath}
\end{theorem}
\begin{proof}
  As $A$ runs over symmetric $n\times n$ matrices with non-negative integer entries summing to $d$, the vectors
  \begin{displaymath}
    v_A = \prod_{i=1}^n e_i^{a_{ii}} \prod_{i<j} (e_i\wedge e_j)^{a_{ij}}
  \end{displaymath}
  form a basis of $W$ consisting of eigenvectors for the action of\linebreak $\Delta(x_1,\dotsc,x_n)$.
  When $A\in M^{\mathrm{sym}}_{\nu\nu}$, the eigenvalue corresponding to $v_A$ is $x^\nu$.
  Now using the correspondence (\ref{eq:rsk-sym}) on symmetric matrices:
  \begin{align*}
    \ch_W(x_1,\dotsc,x_n) &= \sum_{\nu}|\M_{\nu\nu}^{\mathrm{sym}}| x^\nu\\
                          & = \sum_\nu \sum_\lambda |\Tab(\lambda,\nu)| x^\nu\\
                          & =\sum_\lambda s_\lambda(x)\\
                          & = \sum_\lambda \ch_{W_\lambda}(x),
  \end{align*}
  thereby proving the result.
\end{proof}
\begin{remark}[A Gelfand Model]\label{remark:gelfand-model}
  We may say that the symmetric algebra:
  \begin{displaymath}
    \Sym(\CC^n\oplus \wedge^2\CC^n) :=\bigoplus_{k,l\geq 0}\Sym^k(\CC^n)\otimes \Sym^l(\wedge^2 \CC^n)
  \end{displaymath}
  is a \emph{model} for polynomial representation theory of $\gl n$, in the sense of Bernstein, Gelfand and Gelfand \cite{MR0453927}:
  it is a completely reducible representation containing each irreducible polynomial representation of $\gl n$ exactly once.
  For an alternate approach, see \cite[Theorem~4.5.1]{doi:10.1142/9789812562500_0006}.
\end{remark}
The following result was proved for permutation matrices by\linebreak Sch\"utzenberger~\cite[Theorem~4.4]{MR0498826}.
An elegant proof using the light-and-shadows version of the Robinson-Schensted correspondence was suggested (as an exercise) by Viennot \cite[Proposition~4.1]{MR0470059}.
This proof can be adapted to integer matrices using Fulton's matrix ball construction \cite[Section~4.2]{MR1464693}, or the generalization of Viennot's light-and-shadows algorithm in Prasad~\cite[Chapter~3]{rtcv}.
\begin{theorem}
  [Sch\"utzenberger's lemma]
  \label{theorem:schuetzenberger-lemma}
  The RSK correspondence takes symmetric matrices with trace $k$ to semistandard Young tableaux with $k$ odd columns.
\end{theorem}
The term ``odd columns'' alludes to the Young diagram of $\lambda$, whose rows are the parts of $\lambda$.
The columns of $\lambda$ are the parts of the conjugate partition $\lambda'$.
Sch\"utzenberger's lemma gives the following refinement of Theorem~\ref{theorem:rsk-symmetry}:
\begin{theorem}
  \label{theorem:schuetzenberger-refined-model}
  For all non-negative integers $k$ and $l$,
  \begin{displaymath}
    \Sym^k(\CC^n)\otimes \Sym^l(\wedge^2 \CC^n) = \bigoplus_{\lambda\vdash k+2l \text{ with $k$ odd columns }} W^n_\lambda.
  \end{displaymath}
\end{theorem}
Taking $k=0$ gives a better-known special case (see Benson and Ratcliff \cite[Theorem~4.3.1]{doi:10.1142/9789812562500_0006})
\begin{theorem}
  \label{theorem:burge-no-loops}
  For every positive integer $l$,
  \begin{displaymath}
    \Sym^l(\wedge^2 \CC^n) = \bigoplus_{\lambda\vdash 2l \text{with every part appearing even number of times}}W^n_\lambda.
  \end{displaymath}
\end{theorem}
\begin{remark}
  The restriction of the RSK correspondence to symmetric matrices with zeros on the diagonal is the correspondence concerning \emph{graphs without loops} in Burge \cite[Section~2]{burge}.
  Theorem~\ref{theorem:burge-no-loops} can be deduced from it.
\end{remark}
\section{The Burge Correspondence}
\label{sec:burg-corr}
Burge \cite{burge} described four variants of the RSK correspondence.
We begin with a correspondence which, although absent from Burge's paper, is now commonly known as \emph{the Burge Correspondence} \cite[A.4.1]{MR1464693}:
\begin{displaymath}
  \BUR : \M_{\mu\nu}\to \coprod_{\lambda\vdash d} \Tab(\lambda,\nu)\times \Tab(\lambda,\mu).
\end{displaymath}
Indeed, this is a \emph{different} bijection between the same two sets between which the RSK correspondence defines a bijection.
This correspondence shares the symmetry property with the RSK correspondence:
\begin{displaymath}
  \BUR(A) = (P, Q) \text{ if and only if } \BUR(A')=(Q,P),
\end{displaymath}
and consequently induces a bijection:
\begin{equation}
  \label{eq:bur}
  \BUR:\M_{\mu\mu}^{\mathrm{sym}}\tilde\to \coprod_{\lambda\vdash d} \Tab(\lambda,\mu).
\end{equation}
An important difference between the Burge correspondence and the RSK correspondence is that Theorem~\ref{theorem:schuetzenberger-lemma} (Sch\"utzenberger's lemma) fails.
Instead we have \cite[A4.1, Exercise~18]{MR1464693}:
\begin{theorem}[Sch\"utzenberger's lemma for the Burge correspondence]
  \label{theorem:burge-schuetzenberger}
  The Burge correspondence (\ref{eq:bur}) takes a symmetric integer matrix with $k$ odd entries on its diagonal to a semistandard Young tableaux with $k$ odd rows.
\end{theorem}
\begin{theorem}
  \label{theorem:burge}
  For non-negative integers $k$ and $l$, the representation:
  \begin{displaymath}
    W = \wedge^k(\CC^n)\otimes \Sym^l(\Sym^2 \CC^n)
  \end{displaymath}
  of $\gl n$ has decomposition into irreducible representations given by:
  \begin{displaymath}
    W = \bigoplus_{\lambda\vdash k+2l \text{ with $k$ odd rows}} W^n_\lambda.
  \end{displaymath}
\end{theorem}
\begin{proof}
  Given a symmetric integer matrix $A$ with $k$ odd diagonal entries, write $A = A' + A''$, where $A'$ is a diagonal matrix with diagonal entries $0$ or $1$, and $A''$ is an integer matrix with even diagonal entries.
  Write $b_{ii}=a''_{ii}/2$ (half of the $(i,i)$th entry of $A''$) and $b_{ij}=a''_{ij}$ for $i\neq j$.
  Set
  \begin{displaymath}
    v_A = \bigwedge_{a'_{ii}=1} e_i \otimes \prod_{i\leq j} (e_ie_j)^{b_{ij}}.
  \end{displaymath}
  As $A$ runs over symmetric matrices with non-negative integer entries, $k$ odd entries on the diagonal, and off-diagonal entries summing to $2l$, $v_A$ forms a basis of eigenvectors for the action of $\Delta(x_1,\dotsc,x_n)$ on $W$.
  Moreover, when $A\in \M^{\mathrm{sym}}_{\mu\mu}$ then $v_A$ has eigenvalue $x^\mu$.

  Sch\"utzenberger's lemma for the Burge correspondence (Theorem~\ref{theorem:burge-schuetzenberger}) gives:
  \begin{displaymath}
    \ch_W(x) = \sum_{\lambda\vdash k+2l \text{ with $k$ odd rows}} s_\lambda(x),
  \end{displaymath}
  whence the theorem follows.
\end{proof}
When restricted to matrices with even diagonal entries, the correspondence (\ref{eq:bur}) becomes the correspondence in \cite[Section~3]{burge} concerning \emph{graphs with loops and multiple edges}.
In terms of Theorem~\ref{theorem:burge}, this is the special case where $k=0$:
\begin{theorem}
  \label{theorem:symsym}
  For every non-negative integer $k$,
  \begin{displaymath}
    \Sym^k(\Sym^2\CC^n) = \bigoplus_{\lambda\vdash 2k \text{with all parts even}} W^n_\lambda.
  \end{displaymath}
\end{theorem}
\begin{remark}
  \label{remark:howe-symsym}
  Theorem~\ref{theorem:symsym} appears in Howe \cite[Theorem~3.1]{MR1321638} with a different proof.
\end{remark}
\begin{remark}[Another Gelfand model]
  We get a second model (see Remark~\ref{remark:gelfand-model}) for the polynomial representation theory of $\gl n$, namely:
  \begin{displaymath}
    \bigoplus_{k,l\geq 0} \wedge^k \CC^n \otimes \Sym^l(\Sym^2\CC^n).
  \end{displaymath}
\end{remark}
\section{Two more correspondences of Burge}
\label{sec:two-more}
Besides the restrictions of RSK and BUR to symmetric matrices with zeros on the diagonal, the article of Burge \cite{burge} contains two more correspondences.
Let $\N^{\mathrm{sym}}_{\mu\mu}$ denote matrices in $\M^{\mathrm{sym}}_{\mu\mu}$ with entries in $\{0,1\}$.
Let $\N_{\mu\mu}^{\mathrm{sym}, \mathrm{tr}=0}$ denote the subset of $\N^{\mathrm{sym}}_{\mu\mu}$ consisting of matrices with trace zero.
Given a partition $\lambda$, for each $i$ such that $\lambda_i\geq i$, let $\alpha_i=\lambda_i-i$, and $\beta_i = \lambda'_i-i$.
The largest index $d$ such that $\lambda_d\geq d$ is called the \emph{Durfee rank} of $\lambda$, and $(\alpha_1,\dotsc, \alpha_d|\beta_1,\dotsc,\beta_d)$ are called the \emph{Frobenius coordinates} of $\lambda$.
\begin{definition}
  [Threshold Partition]
  A partition $\lambda$ is called a \emph{threshold partition} if its Frobenius coordinates satisfy $\beta_i=\alpha_i+1$ for $i=1,\dotsc,d$, where $d$ is the Durfee rank of $\lambda$.
\end{definition}
Let $\TP(n)$ denote the set of threshold partitions of $n$.
Let $\mu$ be a partition of $d$.
The correspondences of Burge concerning \emph{graphs without loops or multiple edges} \cite[Section~4]{burge} and \emph{graphs without multiple edges} \cite[Section~5]{burge} are:
\begin{align*}
  \BUR_1:& N_{\mu\mu}^{\mathrm{sym},\text{tr}=0}\tilde\to \coprod_{\lambda\in \TP(d)} \Tab(\lambda,\mu)\\
  \BUR_2:& N_{\mu\mu}^{\mathrm{sym}}\tilde\to \coprod_{\lambda'\in \TP(d)} \Tab(\lambda,\mu).
\end{align*}
Their significance in terms of the representation theory of $\gl n$ are given by:
\begin{theorem}
  \label{theorem:burge2}
  For every non-negative integer $d$, we have:
  \begin{align*}
    \wedge^d(\wedge^2\CC^n) &\cong \bigoplus_{\lambda\in \TP(d)} W^n_\lambda,\\
    \wedge^d(\Sym^2\CC^n) & \cong \bigoplus_{\lambda'\in \TP(d)} W^n_\lambda.
  \end{align*}
\end{theorem}
\begin{proof}
  Given a symmetric matrix $A$ with entries in $\{0,1\}$, define
  \begin{displaymath}
    v_A = \bigwedge_{\{(i,j)\mid i<j,\;a_{ij}=1\}} e_i\wedge e_j,
  \end{displaymath}
  taking the terms in lexicographic order.
  As $A$ runs over symmetric matrices with entries in $\{0,1\}$, zeros on the diagonal, and all entries summing to $d$, $v_A$ forms a basis of $\wedge^d(\wedge^2\CC^n)$.
  Moreover, if $A\in \N_{\mu\mu}^{\mathrm{sym},\mathrm{tr}=0}$, $v_A$ is an eigenvector for the action of $\Delta(x_1,\dotsc,x_n)$ with eigenvalue $x^\mu$.
  Now using the bijection $\BUR_1$, it is easy to see that the characters of both sides in the first identity in Theorem~\ref{theorem:burge2} are equal.
  The proof of the second identity is similar.
\end{proof}
\section{Representations of Symmetric Groups}
\label{sec:repr-symm-groups}
If $W_1$ and $W_2$ are polynomial representations of $\gl n$ of degree $d$, and $T:W_1\to W_2$ is a homomorphism of representations, then for each $\mu$, $T(W_1(\mu))\subset W_2(\mu)$.
Here $W_i(\mu)$ denotes the $\mu$ weight space of $W_i$ (Definition~\ref{definition:weight-vector}).
The symmetric group $S_n$ may be regarded as a subgroup of $\gl n$ via permutation matrices.
Given a polynomial representation $W$ of $\gl n$ of \emph{degree $n$}, let $R_n(W)=W(1,\dotsc,1)$.
Clearly, the action of $S_n$ leaves $R_n(W)$ invariant, making it a representation of the symmetric group.
Moreover, for a morphism $T:W_1\to W_2$ of polynomial representations of $\gl n$, let $R_n(T):R_n(W_1)\to R_n(W_2)$ denote the linear map obtained by restriction of $T$ to $R_n(W)$.
We have a functor:
\begin{displaymath}
  R_n: \Rep^n(\gl n)\to \Rep(S_n),
\end{displaymath}
where $\Rep^n(\gl n)$ denotes the category of homogeneous polynomial representations of $\gl n$ of degree $n$, and $\Rep(S_n)$ denotes the category of complex representations of $S_n$.
Given a partition $\lambda=(\lambda_1,\dotsc,\lambda_l)$ of $n$, define:
\begin{displaymath}
  X_\lambda = \{(S_1,\dotsc,S_l)\mid S_1\sqcup \dotsb \sqcup S_l = \{1,\dotsc,n\}, |S_i|=\lambda_i\}.
\end{displaymath}
The action of $S_n$ on $\{1,\dotsc,n\}$ induces an action on $X_\lambda$.

Recall a characterization of irreducible complex representations of $S_n$:
\begin{theorem}
  \label{theorem:characterization}
  \cite[Section~4.4]{rtcv}
  For every partition $\lambda$ of $n$, there exists a unique representation $V_\lambda$ of $S_n$ which occurs in both $K[X_\lambda]$ and $K[X_{\lambda'}]\otimes \epsilon$.
  As $\lambda$ runs over all partition of $n$, $V_\lambda$ forms a set of representatives of isomorphism classes of irreducible representations of $S_n$.
\end{theorem}
\begin{theorem}
  For every partition $\lambda$ of $n$, $R_n(W_\lambda)\cong V_\lambda$.
\end{theorem}
\begin{proof}
  Every submultiset of $\{1,\dotsc,n\}$ is of the form\linebreak $I=\{1^{m_1},\dotsc,n^{m_n}\}$, where $m_i\geq 0$ denotes the multiplicity of $i$ in $I$.
  The size of $I$ is, by definition, the sum $m=m_1+\dotsb+m_n$.
  Then $e_I=\prod_{i=1}^ne_i^{m_i}$ may be regarded as a tensor in $\Sym^m \CC^n$.
  Let $J(n,m)$ denote the set of all submultisets of $n$ of size $m$.
  For $S=(S_1,\dotsc,S_l)\in \prod_{k=1}^l J(n,\lambda_k)$, define
  \begin{displaymath}
    e_S = e_{S_1}\otimes \dotsb \otimes e_{S_l}\in \Sym^\lambda \CC^n.
  \end{displaymath}
  Then as $S$ runs over elements of $\prod_{i=k}^l J(n,\lambda_k)$, $e_S$ forms a basis of $\Sym^\lambda \CC^n$.
  The basis vector $v_S$ has weight $\mu=(\mu_1,\dotsc,\mu_n)$, where $\mu_i$ is the sum of the multiplicities of $i$ in $S_1,\dotsc,S_l$.
  Thus $v_S$ has weight $(1,\dotsc,1)$ if and only if the submultisets $S_1,\dotsc,S_l$ are actually subsets (each element has multiplicity $0$ or $1$), and moreover, form a partition of $\{1,\dotsc,n\}$.
  In other words, $R_n(\Sym^\lambda \CC^n)$ is spanned by $\{v_S\mid S\in X_\lambda\}$.
  It follows that $R_n(\Sym^\lambda \CC^n)\cong K[X_\lambda]$.
  Similarly, $R_n(\wedge^{\lambda'}\CC^n)\cong K[X_{\lambda'}]\otimes \epsilon$.
  Since $W_\lambda$ occurs in $\Sym^\lambda \CC^n$ and also in $\wedge^{\lambda'}\CC^n$, it follows that $R_n(W_\lambda)$ occurs in $K[X_\lambda]$ and $K[X_{\lambda'}]\otimes \epsilon$.
  Therefore, by Theorem~\ref{theorem:characterization}, $R_n(W_\lambda)$ would have to be isomorphic to $V_\lambda$, unless $R_n(W_\lambda)=0$.
  But $\dim(R_n(W_\lambda))$ is the coefficient of $x_1\dotsb x_n$ in $s_\lambda(x_1,\dotsc,x_n)$, which is the number of tableaux of shape $\lambda$ and weight $(1,\dotsc,1)$ (the \emph{standard} tableaux), which is always positive.
\end{proof}
\begin{theorem}
  [Schur-Weyl duality]
  Let $S_m$ act on $\otimes^m \CC^n$ by permuting the tensor factors, and $\gl n$ act on each tensor factor.
  Then, as a representation of $S_m\times \gl n$,
  \begin{displaymath}
    \otimes^m \CC^n \cong \bigoplus_{\lambda\in \Lambda_n^m} V_\lambda\otimes W_\lambda.
  \end{displaymath}
\end{theorem}
\begin{proof}
  Apply the functor $R_m\otimes \mathrm{id}$ to $(\gl m, \gl n)$ duality (Theorem~\ref{theorem:rsk}).
  The $(1^m)$-weight space of $\Sym^m(\CC^m\otimes \CC^n)$ is spanned by vectors of the form $v_A$, where each of the $m$ rows of $A$ sums to $1$.
  If the $1$ in the $i$th row occurs in the $j_i$th column, then map $v_A$ to the vector $e_{j_1}\otimes \dotsb \otimes e_{j_m}\in \otimes^m\CC^n$.
  This induces an isomorphism $(R_m\otimes \mathrm{id})(\CC^m\otimes \CC^n)\tilde\to \otimes^m \CC^n$ of $S_m\times \gl n$-representations.
\end{proof}
\begin{remark}
  Schur-Weyl duality was used by Schur~\cite{schur} to give a second proof of the classification of irreducible polynomial representations of $\gl n$.
  It was popularized by Hermann Weyl \cite{Weyl} and is known as Schur-Weyl duality.
  See also \cite{MR2349209} and \cite[Section~6.4]{rtcv}.
\end{remark}
\begin{theorem}
  [Gelfand models for $S_n$]
  \label{theorem:gelfand-sn}
  Let $M_{n,k}$ denote the set of all elements $w\in S_n$ such that $w=w^{-1}$ and $w$ has $k$ fixed points.
  Define a representation of $S_n$ on $\CC[M_{n,k}]$ (the space of complex-valued functions on $M_{n,k}$) by:
  \begin{displaymath}
    \rho_1(g)f(w) = (-1)^{i_1(g,w)}f(g^{-1}wg)
  \end{displaymath}
  where $i_1(g,w)$ is the number of indices $i<j$ such that $w(i)=j$ and $g(i)>g(j)$.
  Then
  \begin{displaymath}
    \CC[M_{n,k}] = \bigoplus_{\{\lambda\vdash n\mid \text{ with $k$ odd columns\}}} V_\lambda.
  \end{displaymath}
  In particular, if $M_n=\coprod M_{n,k}$ is the set of all elements $w\in S_n$ such that $w^2=\mathrm{id}$, then $\CC[M_n]$ is a Gelfand model for $S_n$.
\end{theorem}
\begin{proof}
  Apply the functor $R_n$ to both sides of the first identity in Theorem~\ref{theorem:schuetzenberger-refined-model}.
\end{proof}
\begin{remark}
  The above Gelfand model for symmetric groups is well-known.
  For an alternative approach see \cite{Inglis1990,kodiyalam-verma}.
  The decomposition in Theorem~\ref{theorem:burge} also gives a Gelfand model for $S_n$; the model obtained is the twist of the model in Theorem~\ref{theorem:gelfand-sn} by the sign character.
  Note that the twist of $V_\lambda$ by the sign character is $V_{\lambda'}$.
\end{remark}
\begin{remark}
  [Relation to induced representations]
  The representation $\rho_1$ of $S_n$ on $\CC[M_{n,k}]$ in Theorem~\ref{theorem:gelfand-sn} can be viewed as an induced representation.
  For each positive integer $l$, let $B_{2l}\subset S_{2l}$ denote the centralizer of the involution $(1,2)(3,4)\dotsb(2l-1,2l)\in S_{2l}$.
  Let $e_{2l}$ denote the restriction of the sign character of $S_{2l}$ to $B_{2l}$.
  If $n=2l+k$, then $B_{2l}\times S_k$ is a subgroup of $S_n$.
  We have $\rho_1 = \Ind_{B_{2l}\times S_k}^{S_n}e_{2l}\otimes 1_{S_k}$, where $1_{S_k}$ denotes the trivial character of $S_k$.
  Taking $k=0$ gives:
  \begin{displaymath}
    \Ind_{B_{2l}}^{S_{2l}} e_{2l} = \bigoplus_{\{\lambda \vdash 2l\mid \text{ every part appearing even number of times}\}} V_\lambda.
  \end{displaymath}
  Similarly, applying the functor $R_{2l}$ to both sides of Theorem~\ref{theorem:symsym} gives:
  \begin{displaymath}
    \Ind_{B_{2l}}^{S_{2l}} 1_{B_{2l}} = \bigoplus_{\{\lambda\vdash 2l\mid \text{ with all parts even}\}} V_\lambda.
  \end{displaymath}
\end{remark}
Applying the functor $R_n$ to the first identity in Theorem~\ref{theorem:burge2} gives an interesting multiplicity-free representation of $S_n$:
\begin{theorem}
  \label{theorem:Sn-burge2}
  Let $n$ be an even integer, say $n=2m$.
  Given pairs $(i,j)$ and $(k,l)$, let $(i',j')$ and $(k',l')$ denote their sorted versions, so that $i'<j'$ and $k'<l'$.
  Write $(i,j)<(k,l)$ if $(i',j')<(k',l')$ in lexicographic order.
  Define a representation of $S_n$ on $\CC[M_{n,0}]$ by:
  \begin{displaymath}
    \rho_2(g)f(w) = (-1)^{i_2(g,w)}f(g^{-1}wg),
  \end{displaymath}
  where
  \begin{multline*}
    i_2(g,w)=i_1(g,w)+\#\{(i,j)<(k,l)\mid i<j, k<l, w(i)=j, w(k)=l,\\(g(i),g(j))>(g(k),g(l))\}.
  \end{multline*}
  Then
  \begin{displaymath}
    \CC[M_{n,0}] = \bigoplus_{\lambda\in \TP(n)} V_\lambda.
  \end{displaymath}
\end{theorem}
\begin{remark}
  [Relation to induced representations]
  The representation $\rho_2$ of $S_{2m}$ on $M_{2m,0}$ in Theorem~\ref{theorem:Sn-burge2} can be viewed as an induced representation.
  The group $B_{2m}$ is a semidirect product $(\ZZ/2\ZZ)^m\rtimes S_m$.
  The $\eta_m$ be the character of $B_{2m}$ whose restriction to $S_m$ is the sign character of $S_m$ and whose restriction to $(\ZZ/2\ZZ)^m$ is the product of the non-trivial characters of each of the $m$ factors.
  Then Theorem~\ref{theorem:Sn-burge2} can be restated as:
  \begin{displaymath}
    \Ind_{B_{2m}}^{S_{2m}}\eta_m = \bigoplus_{\lambda\in \TP(n)} V_\lambda.
  \end{displaymath}
  For a different approach to this result see \cite[Exercise~45.6]{bump}
\end{remark}
When $R_n$ is applied to the second identity in Theorem~\ref{theorem:burge2}, the representations obtained are twists by the sign character of the representations in Theorem~\ref{theorem:Sn-burge2}.

\end{document}